\documentclass[12pt]{article} 

\usepackage[utf8]{inputenc} 

\usepackage{geometry}
\usepackage{amsmath,amsthm,amsfonts,amssymb}
\usepackage{textcomp,enumerate,bbm,latexsym, bm,color}
\usepackage{graphicx} 
\usepackage{url,algorithm,algorithmic}
\usepackage[authoryear]{natbib}
\usepackage{multirow}
\usepackage{booktabs}
\usepackage{cleveref}
\usepackage[english]{babel}
\usepackage{subcaption}

\def\spacingset#1{\renewcommand{\baselinestretch}%
{#1}\small\normalsize} \spacingset{1}

\newcommand{\cf}{\mathcal{F}}
\newcommand{\real}{\mathbb{R}}
\newcommand{\dnorm}{\mathcal{N}}

\renewcommand{\emptyset}{\varnothing}
\renewcommand{\Pr}{\mathbb{P}}
\newcommand{\e}{\mathbb{E}}
\newcommand{\wt}{\widetilde}

\newcommand{\bsx}{\boldsymbol{x}}
\newcommand{\bsy}{\boldsymbol{y}}
\newcommand{\bsz}{\boldsymbol{z}}

\newcommand{\bsP}{\boldsymbol{P}}
\newcommand{\bsp}{\boldsymbol{p}}

\newcommand{\bstheta}{\boldsymbol{\theta}}

\newcommand{\mon}{\mathrm{mon}}

\theoremstyle{plain}
\newtheorem{thm}{Theorem}[section]
\newtheorem{lem}[thm]{Lemma}

\newtheorem{prop}[thm]{Proposition}

\theoremstyle{definition}
\newtheorem{defn}{Definition}
\newtheorem{exmp}{Example}

\newtheorem{assumption}{Assumption}

\crefname{assumption}{assumption}{assumptions}
\crefname{exmp}{example}{examples}

\theoremstyle{remark}
\newtheorem{rem}{Remark}[section]

\usepackage{fancyhdr} 
  \title{\bf Admissibility in partial conjunction testing}
  \author{Jingshu Wang \\Stanford University \and Art B. Owen\\Stanford University}
\date{December 2016}

\begin{document}

  \footnotetext{This work was supported by the US National
Science Foundation under grant DMS-1407397.}

  \maketitle

\begin{abstract}
Meta-analysis combines results from multiple studies aiming to 
increase power in finding their common effect. It would typically reject 
the null hypothesis of no effect if any one of the studies shows strong 
significance.  
The partial conjunction null hypothesis is rejected only when at least $r$ of
$n$ component hypotheses are non-null with $r=1$
corresponding to a usual meta-analysis. 
Compared with meta-analysis, it can  
encourage replicable findings across studies. 
A by-product of it when applied to different $r$ values is a 
confidence interval of $r$ quantifying the proportion of 
non-null studies. 
Benjamini and Heller (2008) provided a valid test for the partial conjunction 
null by ignoring the $r-1$ smallest p-values and applying a 
valid meta-analysis p-value to the remaining $n-r+1$ p-values.
We provide sufficient and necessary conditions of admissible 
combined p-value for the partial conjunction hypothesis 
among monotone tests. Non-monotone tests always dominate monotone tests 
but are usually too unreasonable to be used in practice. 
Based on these findings, we propose a generalized form of Benjamini and 
Heller's test which allows usage of various types of meta-analysis p-values, 
and apply our method to an example in assessing replicable 
benefit of new anticoagulants across subgroups of patients for stroke prevention.  
\end{abstract}


\section{Introduction}\label{sec:introduction}

When a null hypothesis is tested in $n$ different settings,
a meta-analysis can be used to obtain a combined p-value
based on all of the test results. It gains power as 
the combined p-value is usually more 
significant than any of the individual p-value in each setting. 
However, the combined p-value in meta-analysis is only 
valid for the global null where  
the null hypothesis is true in every setting, thus 
it is possible that the null is then rejected
largely on the basis of just one extremely significant component
hypothesis test.  Such a rejection may be undesirable as it
could arise from some irreproducible property of the setting
in which that one component test was made. 

\begin{table}\label{table:puzzles}
\caption{\label{tab:examples}
four hypothetical cases for five ordered p-values.}
\centering 
\begin{tabular}{clllll}
Case & $p_{(1)}$ & $p_{(2)}$ & $p_{(3)}$ & $p_{(4)}$ & $p_{(5)}$\\
\midrule 
A & $10^{-200}$ & $0.4$ & $0.5$ & $0.6$ & $0.7$\\
B & $10^{-10}$ & $10^{-9}$ & $10^{-8}$ & $10^{-7}$ & $10^{-6}$\\
\midrule 
C & $10^{-100}$ & $10^{-100}$ & $10^{-100}$ & $0.049$ & $0.8$\\
D & $0.048$& $0.048$& $0.048$& $0.048$ & $0.8$\\
\end{tabular}
\end{table}

Refering to \Cref{tab:examples}, cases A and B illustrate the potential problem 
of meta-analysis. 
Both a Fisher and a Stouffer meta-analysis
would find case A more significant than case B, 
while the only significant setting in case A may largely due to a technical 
or statistical bias. The random effect model in meta-analysis has been widely 
accepted for consideration of heteogeneity across studies \citep{higgins2009}. 
However, it still 
assumes that the effects across studies are similar and does not explicitly 
guarantee replication nor being robust to extreme bias.

Researchers in functional magnetic resonance imaging (fMRI)
have adopted conjunction (logical `and') testing 
\citep{pric:fris:1997, friston1999multisubject, nichols2005} in which a
hypothesis must be rejected in all $n$ settings where it is tested.
The $n$ settings may correspond to related tasks or they may
correspond to independent subjects. 
For example in \Cref{tab:examples}, 
a conjunction test would only reject case B which shows 
consistent replication. 
However, conjunction tests lose power for large $n$ as they are based on the
largest of $n$ p-values.
A compromise is to require evidence that at least $r$ out of $n$ null hypotheses
are false, for some user specified $r$. Such tests of the `partial
conjunction (PC) null hypothesis' were used in \cite{fris:penn:glas:2005}
and then studied by \cite{benjamini2008screening}.  The extremes $r=1$
and $r=n$ correspond to the usual meta-analysis tests and conjunction
testing respectively.


Partial conjunction testing 
is useful in areas beyond neuroimaging. It is the tool for replicability analysis, 
which finds effects that present in more than one studies, and has been 
applied in systematic 
review for healthcare prevention  \citep{shenhav2015} and genome-wide association 
studies \citep{heller2014}. PC test also has the potential usage in 
finding common gene regulation patterns across tissues for eQTL data 
\citep{flutre2013}. 
Besides, it can be applied to gene set enrichment analysis to avoid  
selection of gene sets whose significance depend on only one single 
gene \citep{wang2010}.


A Benjamini-Heller partial conjunction (BHPC) test works as follows. 
One sorts the observed
p-values yielding $p_{(1)} \le p_{(2)} \le \dots \le p_{(n)}$, ignores the smallest
$r-1$ of them, and then applies a 
valid p-value combination rule to the remaining $n-r+1$ p-values.
\cite{benjamini2008screening} show that BHPC tests are valid for the
partial conjunction null when the $n$ hypotheses are
independent. They also consider some dependent test conditions as well
as the consequences of using PC tests in the Benjamini-Hochberg procedure. 

Cases C and D illustrate an interesting property of the BHPC tests.
Suppose that we need to reject at least four null hypotheses to have
a meaningful finding.  Then a BHPC test finds that case D is stronger evidence
(smaller p-value) than case C, because BHPC is based only on $p_{(4)}$
and $p_{(5)}$.  In case C we are extremely confident of three rejections and 
are banking on the fourth one to be correct.
In case D by contrast, none of the four smallest p-values
is much better than borderline.  It appears to have about four times
as many ways to disappointing us.  
This comparsion between case C and D reveals  
a counter-intuitive property of the BHPC tests, that we study further.

Here we investigate the power properties of BHPC tests focussing on
admissibility. 
Under the assumption that 
the component p-values are either independent or have a positive dependency 
structure, we 
characterize the complete class of tests for monotone admissibility, 
which is a generalized form of BHPC p-values (GBHPC p-values). 
the only admissible PC tests 
among monotone tests are 
either of the BHPC form, or its generalization (GBHPC),   
which uses combined p-values constructed by taking the maximum of 
the meta-analysis p-value of each of the $n \choose {r - 1}$ subsets of $n - r + 1$ hypotheses. 
Under mild assumptions, a sufficient 
condition for the monotone admissibility of a GBHPC p-value is that 
each of the meta-analysis p-value for $n \choose {r - 1}$ subsets is 
admissible. 
GBHPC p-values are also called r-values in \cite{shenhav2015}. 

The monotonicity condition, which means that the combined p-value is a 
non-decreasing function of the individual p-values, 
is necessary for us to discuss admissibility for partial conjunction hypotheses with 
$r > 1$.
If we relax this condition, then BHPC tests
become inadmissible. Because non-monotone tests are quite 
unreasonable scientifically in most cases, this is not a strong criticism of BHPC.
We side with \cite{perlman1999emperor} in rejecting the admissibility
criterion not the test, when methods lacking face-value validity are
included in comparisons.

The admissibility properties show supreme of BHPC p-values, but 
how can we then explain its puzzling behavior alluded to for cases C and 
 D in Table~\ref{tab:examples}? An explanation is that unlike the combined p-values in 
 meta-analysis, the PC p-values measure the strength of replicability (the true 
 proportion of non-null studies) instead of the magnitude of effect size. 
 a PC p-value can be much smaller when the true number ($r_0$) 
 rather than the effect size of non-null studies is large. Compared with case D,  
 the three extreme p-values of case C in Table~\ref{tab:examples} gain us 
 much stronger evidence of a large effect size but not much more evidence 
 on $r_0 \geq 4$. Thus, the PC p-values for both case C and case D are 
 similar.

%



\Cref{sec:preliminaries} presents our notation and some background on partial conjunction tests and admissibility.
\Cref{sec:combining} proposes the GBHPC p-values and 
presents the main theorems on monotone admissible partial conjunction 
p-values. 
\Cref{sec:simulation} is a simulation study for the power comparison 
of several GBHPC p-values under various hypotheses configurations. 
Compared with BHPC p-values, GBHPC p-values have the advantage that it 
can be constructed from more sophisticated meta-analysis p-values. 
We illustrate this benefit in \Cref{sec:real} in an application of 
GBHPC p-values 
for assessing replicable 
benefit and safety concerns of new oral anticoagulants  
across subgroups of patients for stroke prevention. 
%
\Cref{sec:conclusion} has our conclusions and states some future work.

\section{Preliminaries}\label{sec:preliminaries}

\subsection{Definitions and notations}
The problem begins with $n$ null hypotheses to test,
$H_{0i}$ for $i=1,\dots,n$. 
Each $H_{0i}$ is the hypothesis to test for an individual setting/study.
The corresponding alternative 
hypotheses are $H_{1i}$.
The $i$'th hypothesis refers to a parameter $\theta_i$.
If $H_{0i}$ holds then $\theta_i\in\Theta_{0i}$, while $H_{1i}$
specifies that $\theta_i\in\Theta_{1i}$. The parameter space
for the $i$'th hypothesis is $\Theta_i=\Theta_{0i}\cup\Theta_{1i}$
and of course $\Theta_{0i}\cap\Theta_{1i}=\emptyset$. 
The parameter space of $(\theta_1, \dots, \theta_n)$ is 
$\Theta = \prod_i \Theta_i$.

To each hypothesis, there corresponds a p-value, $p_i$.
there may be a loss of information in reducing a
data set to one p-value. Yet often that loss is small
and very commonly the researchers who gathered the original
data share only their p-values for reasons that may include
privacy of their subjects.

We use $p_i$ to denote the numerical value of the p-value
for the $i$'th hypothesis.  It is the observed value of a corresponding
random variable $P_i$.  
The sorted p-values are $p_{(1)}\le p_{(2)}\le\dots\le p_{(n)}$
and $P_{(1)}\le P_{(2)} \le \dots\le P_{(n)}$ are the sorted random variables.
Probability and expectation for functions of $P_i$ are given 
by $\Pr_{\theta_i}$ and $\e_{\theta_i}$ respectively. 
We let $\theta=(\theta_1,\dots,\theta_n)$ and
$\bsP=(P_1,\dots,P_n)$. Probability and expectation for functions of $\bsP$
are given by $\Pr_{\theta}$ and $\e_{\theta}$. Each 
$p_i$ is a valid P-value according to the definition below:

\begin{defn}[Validity]\label{def:validp}
A valid component p-value satisfies $\sup_{\theta_i\in\Theta_{0i}}\Pr_{\theta_i}(P_i\le \alpha)\le \alpha$ for $0\le\alpha\le1$.
\end{defn}

Besides independency, positive dependency can be a common dependency structure across studies, especially 
when they share samples.
For  
$P_1,\dots,P_n$, we assume that 
they are either independent or positively dependent (PRDS) \citep{benjamini2001} 
Under any parameter $\theta \in \Theta$. 
Using the definition of PRDS
thus they have the following property: for $\forall (p_1, \cdots, p_n) \in [0, 1]^n$ 
and $\forall \theta \in \Theta$,
\begin{equation}\label{eq:prds}
\Pr_{\theta}(P_1 \leq p_1, \cdots, P_n \leq p_n) \geq \prod_i 
\Pr_{\theta_i}(P_i \leq p_i)
\end{equation}

For a given $r \leq n$, the PC null hypothesis and alternative hypothesis are defined 
as
\begin{align*}
H_0^{r/n}:\quad & \{\text{at most } r - 1 \text{ hypotheses are non-null}\},\quad\text{and}\\
H_1^{r/n}:\quad & \{\text{at least } r \text{ hypotheses are non-null}\}.
\end{align*}
The null space is defined as $\Theta_0^{r/n} = \{\theta\in \Theta: H_0^{r/n} \text{ is true}\}$.


We use $1{:}r$ to denote $\{1,2,\dots,r\}$ and similarly
$(r+1){:}n=\{r+1,r+2,\dots,n\}$.
The index set $u\subset 1{:}n$ has cardinality $|u|$
and complement $-u=1{:}n\setminus u$.
Under the null hypothesis $H_{0u}$ we have $\theta_j\in\Theta_{0j}$
for all $j\in u$. The null space of $H_{0u}$ is denoted as $\Theta_{0u}$.

Sometimes we combine points $\bsx\in\real^n$ and $\bsy\in\real^n$
into a point $\bsz\in\real^n$ with $z_j=x_j$ for $j\in u$ and $z_j=y_j$
for $j\not\in u$.
Such a hybrid point is denoted $\bsz=\bsx_{u}{:}\bsy_{-u}$. 
Let $u = \{i_1, i_2, \dots, i_{k}\}$, then 
$\theta_u$ is defined as 
the combination $(\theta_{i_1}, \theta_{i_2}, \dots, \theta_{i_k})$.

We can extend the definition of validity to meta-analysis and PC 
p-values.  The combination of $k$ p-values
($k$ may differ from $n$ later)
produces the combined p-value
$p_{r/k}=f_{r,k}(P_1,\dots,P_k)$ which is a valid p-value for
testing $H_0^{r/k}$ if 
$$\sup_{\theta \in \Theta_0^{r / n}} \Pr_\theta( P_{r/k}\le\alpha)\le\alpha,\quad
\forall\, 0\le\alpha\le1.
$$

\begin{defn}[Sensitivity] A sensitive p-value 
  $P_{r/k}=f_{r,k}(P_1,\dots,P_k)$ for $H_0^{r/n}$ satisfies $$\liminf_{\bsP_u \to 
  \bm{0}}P_{r/k} = 0$$
  for $\forall u \subset 1{:}k$ and $|u| = r$.
\end{defn}

Sensitivity requires that PC p-value drops to $0$ when we are certain to 
reject any of $r$ individual hypotheses. For meta-analysis p-value $P_{1/k}$, it 
means that we reject the global null when we are certain at any of 
the individual hypothesis. We think that it is a 
practically reasonable requirement for 
a p-value for testing $H_0^{r/k}$.

Here are some examples of valid and sensitive meta-analysis p-values
given valid p-values $p_1,\dots,p_k$. The combination for a method
$M$ is defined in terms of a function $f_{M,k}$ which may incorporate
sorting of its arguments.

\begin{exmp}\label{example:simes}
Simes' method:
$$p_{S,k} = f_{S,k}(p_1,\dots,p_k)\equiv\min_{i = 1, \cdots, k} \biggl\{ \frac{kp_{(i)}}{i}  \biggr\}.$$
\end{exmp}
\begin{exmp}\label{example:fisher}
Fisher's method:
$$
    p_{F,k} = f_{F,k}(p_1,\dots,p_k)\equiv \mathbbm{P} \biggl( \chi_{(2k)}^2 \geq - 2 
      \sum_{i = 1}^k \log p_{i} \biggr).
$$
\end{exmp}

  \begin{exmp}\label{exmp:wtstouf}
Weighted Stouffer test:
Consider test statistics $T_i \sim \dnorm(\sqrt{n_i}\theta_i/\sigma_i, 1)$, with
sample sizes $n_i$ for $i=1,\dots,k$ and known $\sigma_i>0$. The p-value for the null
that $\theta_i=0$ versus the alternative that $\theta_i>0$ is 
$p_i=1-\Phi(T_i)=\Phi(-T_i)$.
A weighted Stouffer p-value for $H_0^{1/k}$ takes the form
$$
p_{\mathrm{WS},k}  = p_{\mathrm{WS},k}(p_1,\dots,p_k) 
\equiv 1 - \Phi\Biggl(\frac{\sum_{i=1}^{k} \sqrt {n_{i}}\Phi^{-1}(1 - p_{i})/\sigma_i}{\sqrt{\sum_{i=1}^{k}{n_{i}}/\sigma_i^2}}\Biggr).
$$
In fact, $p_{\mathrm{WS}, k}$ can be used beyond one-sided tests. For example, 
for two-sided test $\Phi^{-1}(1 - p_i)$ also has  
the magnitude roughly proportional to $\sqrt{n_i}$. 
We shall illustrate the performance under such usage in our simulations in 
\Cref{sec:simulation}.
\end{exmp}

  \begin{exmp}\label{exmp:TPM}
Truncated product method \citep{zaykin2002}:
this is a more recently developed method to gain efficiency in the presence of 
outliers. The test statistic has the form 
\[
w_\gamma = \prod_i p_i^{1_{p_i \leq \gamma}}
\]
where $\gamma$ is some pre-determined value. 
The TPM p-value for $H_0^{1/k}$ takes the form
$$
p_{\mathrm{TPM},k}  = \Pr(W \leq w_\gamma).
$$
where the probability function of $W$ was computed for both independency and dependent 
scenarios in \citep{zaykin2002}.
\end{exmp}

For non-symmetric meta-analysis p-values, there are also 
weighted
versions of Fisher tests.
Note that each of the functions $f$ in the previous examples  
is monotone according to this definition:
\begin{defn}(Monotonicity)
the p-value $f(p_1,\dots,p_k)$ is monotone if 
the function $f$ is non-decreasing in each argument. 
The set of such monotone p-value functions is denoted $\cf_{\mon}$. 
a monotone test is one that rejects its null hypothesis for small 
values of a monotone p-value. 
\end{defn}

A non-monotone test would reject its null hypothesis 
at some input $(p_1,\dots,p_k)$ but fail to reject at 
some $(p_1',\dots,p_k')$  with all $p_i'\le p_i$. 
Such a test is typically unreasonable.

Besides monotonicity, we also clearify the definition of a 
symmetric combined p-value:
\begin{defn}(Symmetry)
The combined p-value $f(p_1,\dots,p_k)$ is symmetric 
if its value stays unchanged under any permutation of $\{p_1, \dots, p_k\}$. 
\end{defn}

Finally, we state the concept of admissibility using the definition of 
admissible tests from \citet[Chapter 6.7]{lehmann2006testing}.
Let a hypothesis test of $H_0$ versus $H_1$
described by a function $\varphi(X)\in\{0,1\}$
of the data $X$. If $\varphi(X)=1$
then $H_0$ is rejected and $\varphi(X)=0$ otherwise.
The test $\varphi$ is valid at level $\alpha$ if 
$\sup_{\theta\in\Theta_0}\e_\theta(\varphi(X))\le\alpha$.
In our context, the data are a vector $\bsP=(P_1,\dots,P_n)$
of p-values and $\varphi(P_1,\dots,P_n) = 1_{f(P_1,\dots,P_n)\le\alpha}$
where $f$ is a p-value combination function.

\begin{defn}[$\Psi$,$\alpha$-admissibility] 
The level-$\alpha$ test $\varphi\in\Psi$ is 
$\alpha$-admissible for testing $H_0:\theta \in \Theta_0$ against $H_1 :
  \theta \in \Theta_1$ if for any other level-$\alpha$ test $\varphi'\in\Psi$
  \[ \e_{\theta}(\varphi') \geq \e_{\theta}(\varphi),\quad\text{for all $\theta \in \Theta_1$}  \]
  implies $\e_{\theta}(\varphi') = \e_{\theta}(\varphi)$ for  all $\theta \in \Theta_1$. 
\end{defn}

The definition of admissibility depends on the alternatives in $\Theta_1$
as well as the space $\Psi$ of test functions.  The constraints on $\Theta_1$ are important. 
for the ordinary meta-analysis, 
\cite{birn:1954} shows that every monotone p-value is admissible 
when the component p-values 
are independent and the null hypothesis is simple, 
because there is then some alternative
at which that p-value gives optimal power.  However, those optimizing alternatives
may not all be reasonable. \cite{birnbaum1955characterizations} and \citet{stein1956admissibility} 
(generalized later by \citet{matthes1967tests} to include nuisance parameters) also showed that 
for the ordinary meta-analysis, 
when the test statistic distribution
is an exponential family with $\theta$ as canonical parameter, a necessary and sufficient 
condition for admissibility is to have a closed convex acceptance region of 
underlying test statistics.


For the space of test functions $\Psi$, 
traditionally it contains all possible functions when considering admissibility. 
However, for the partial conjunction null hypothesis, 
we restrict $\Psi$ to only include tests using monotone p-values 
to avoid 
unreasonble more powerful tests (see \Cref{sec:inadmiss} for details).

\subsection{BHPC p-values}
Now we restate Theorem 1 of \cite{benjamini2008screening}.

\begin{thm} \label{thm:BHPC_thm}
Let $P_1,\dots,P_n$ be independent valid p-values, and for $k=n-r+1$
let $f_{M,k}(P_1,\dots,P_k)$ be a valid and symmetric meta-analysis p-value 
where $f_{M,k}\in\cf_\mon$.
Then $P_{r/n} = f_{M,n -r+1}(P_{(r)}, P_{(r + 1)}, \dots, P_{(n)})$
is a valid p-value for $H_0^{r/n}$.
\end{thm}

As mentioned, we call the combined p-value $P_{r / n}$ 
described in \Cref{thm:BHPC_thm} a BHPC p-value for short. 
In practice it makes sense to require that the $p$-value combination function
$f_{M,k}(\cdot)$, for $k=n-r+1$, be a sensitive one for $H_0^{1/k}$.  
Notice that if $f_{M,k}$ were a partial conjunction
test of $H_0^{s/k}$ for $s>1$, then $f_{M, k}$ is still a valid meta-analysis 
p-value but is not sensitive any more. 
The BHPC p-values satisfy the chain rule in the sense that 
 $P_{r/n}$ in Theorem~\ref{thm:BHPC_thm} now becomes 
 a valid test for $H_0^{(r+s-1)/n}$. 
Though $P_{r/n}$ would still be a valid test of $H_0^{r/n}$, it is less 
efficient. 

Based on the relationship between hypotheses testing and building 
confidence set, if we have for each $r = 1, 2, \dots, n$ a valid combined p-value 
$p_{r/n}$ for $H_0^{r/n}$, then the $1- \alpha$ confidence set for $r$ is 
\[
i = \{r: P_{r/n} \leq \alpha\}
\]
\cite{benjamini2008screening} showed that if each $P_{r/n}$ is a BHPC p-value 
with $f_{M, k} = f_{S, k}$ in \ref{example:simes}, then we have 
$p_{1/n} \leq P_{2/n} \leq\dots \leq P_{n/n}$ and 
$i = [\hat r, n]$ becomes a $1 - \alpha$ confidence interval 
with $\hat r = \max\{r: P_{r/n} \leq \alpha\}$.

\section{GBHPC p-values}\label{sec:combining}

Motivated by the BHPC p-value, we discuss a more general class of
combined p-values with good power properties. These are 
defined as GBHPC p-values:

\begin{defn} [GBHPC p-value]\label{def:GBHPC}
For each $u\subset1{:}n$ with $|u|=k=n-r+1$ let $g_u$ be a function from $[0,1]^k$ to $[0,1]$
such that $g_u$ is non-decreasing and is a 
valid meta-analysis p-value for $H_{0u}$.  Then
\begin{align}\label{eq:defgbhpc}
f^\star(\bsp) = f^\star(p_1, \cdots, p_n)= \max_{u\subset1{:}n\atop |u|=n-r+1}g_u(\bsp_u)
\end{align}
is a generalized BHPC (GBHPC) p-value.
\end{defn}

The alternative hypothesis $H_{1}^{r/n}$ that at least $r$ out of $n$ hypotheses are false is equivalent to the statement that for every $n-r+1$ of the $n$ hypotheses, there 
is at least one of them that is false. This is the reason that the GBHPC p-value 
is the maximum of all the meta-analysis p-values of size $n -r +1$. 
Using above explanation, the next proposition states validity of 
GBHPC p-values under any dependency structure of the individual p-values.

\begin{prop}\label{prop:validity}
  Any GBHPC p-value is a valid p-value for $H_0^{r/n}$. 
\end{prop}

\begin{proof}
  Consider a GBHPC p-value of the form \eqref{eq:defgbhpc}. From the definition of $H_0^{r/n}$, for all $\theta \in
  \Theta_0^{r / n}$, there exists $u$ with $|u| = n-r+1$ 
  such that $\theta_{j} \in \Theta_{0j}$ for all $j\in u$. 
then for any $\alpha \in [0, 1]$,
  \begin{align*}
     & \Pr_{\theta} (f^{\star} (\bsP) \leq  \alpha)
    \leq \Pr_{\theta}(g_u(\bsP_u)\le \alpha)=\Pr_{\theta_u}(g_u(\bsP_u)\le \alpha)\le\alpha.
  \end{align*}
    Thus $f^{\star}(P_1, \cdots, P_n)$ is valid for $H_0^{r / n}$.
\end{proof}

The BHPC p-value is a special case of GBHPC p-values. If all 
$g_u \equiv g$ in \eqref{eq:defgbhpc} and $g$ a monotone and symmetric 
combined p-value, then $f^\star(\bsp) = g(p_{(r)}, \dots, p_{(n)})$  
becomes a BHPC p-value. Some meta-analysis methods,
such as the weighted Stouffer test in \Cref{exmp:wtstouf},
treat their component p-values differently
depending on the relative sample sizes on which they are based.
The GBHPC framework includes such methods.

Next we show that a sensitive GBHPC p-value has a unique representation 
in the form of \eqref{eq:defgbhpc}.

\begin{prop}\label{prop:unique}
If a GBHPC p-value $f^\star$ is sensitive, then each $g_u$ is also 
sensitive and the representation of $f^\star$ 
in \eqref{eq:defgbhpc} is unique with 
  \begin{equation}\label{eq:unique_g}
    g_{u} (\bsP_u) = \inf_{\bsP_{-u} \in (0, 1]^{r-1}} f^\star (P_1, \cdots, P_n)
  \end{equation}
\end{prop}

\begin{proof}
Consider any given $u \in 1{:}n$ with $|u| = n -r + 1$. Then for $\forall i \in u$, let 
$u_i = \{-u, i\}$, then $|u_i| = r$. As $f^\star$ is sensitive, we have 
\[
\liminf_{P_i \to 0} g_u(\bsP_u) = \liminf_{\bsP_{u_i} \to \bm{0}} g_u(\bsP_u)
\leq \liminf_{\bsP_{u_i} \to \bm{0}} f^\star(\bsP) = 0
\]
Thus, $g_u$ is sensitive for every $u$ with $|u| = n-r+1$. 
On the other hand, for a given $u_0$ with $|u_0| = n-r+1$, 
using definition \eqref{eq:defgbhpc}, we 
have 
\[
\inf_{\bsP_{-u_0} \in (0, 1]^{r-1}} f^\star (\bsP) = 
\inf_{\bsP_{-u_0} \in (0, 1]^{r-1}} \left[\max_u g_u(\bsP_u)\right] = g_{u_0}(\bsP_{u_0})
\]
this also proves the uniqueness of the representation.
\end{proof}

\begin{rem}\label{rem:bhpc}
Equation \eqref{eq:unique_g} also shows that any symmetric GBHPC p-value 
is a BHPC p-value. If $f^\star$ is symmetric, then $g_u$ constructed in 
\eqref{eq:unique_g} is also valid, monotone and symmetric. Also, $g_u \equiv g$ is the 
same 
for all $u$. Thus 
$f^\star = g(p_{(n)}, \cdots, p_{(r)})$ is a BHPC p-value.
\end{rem}

\begin{rem}
Equation~\eqref{eq:defgbhpc} involves taking the maximum of meta-analysis p-values over 
${n \choose r-1}$ number of subsets. Thus, the computational cost 
of non-symmetric GBHPC p-values can be high for large $r$ and $n$. 
One situation to use a non-symmetric GBHPC p-value is when $r = 2$, then 
${n \choose r-1} = n$ which is typically acceptable. Besides, 
it is sometimes possible to 
make use of the special structure of the problem to construct 
easy-to-compute non-symmetric GBHPC p-values. For 
instance in \Cref{sec:real}, the GBHPC p-values 
we constructed for the pharmaceutical data has almost the same computational cost as BHPC p-values, but can give much smaller p-values.
\end{rem}

\subsection{Monotone $\alpha$-admissiblity}\label{sec:admissibility}

Now we discuss the sufficient and necessary conditions for admissible combined 
p-values of a PC hypothesis. Each of our results uses some combination
of the following three assumptions on the individual p-values.

\begin{assumption}[Strong alternatives]\label{assumption:power}
$\forall \alpha > 0$ and $i = 1, \dots, n$,
  $\sup_{\theta_i \in \Theta_{1 i}} \mathbbm{P}_{\theta_i} (P_i \leq \alpha) = 1$.
\end{assumption}

\begin{assumption}[Continuity]\label{assumption:continuity}
For $\forall \theta_i \in \Theta_{1i}$ and 
$i = 1,  \dots, n$, $\mathbbm{P}_{\theta_i} (P_i = 0)=0$.
\end{assumption}

\begin{assumption}[Completeness]\label{assumption:complete}
 The family $\{\mathbbm{P}_{\theta_{u}}: \theta_u \notin \Theta_{0u}\}$  
 for any subset $u$ with $|u| = n -r +1$ is complete.
\end{assumption}

\Cref{assumption:power} states that for each individual hypothesis 
there are strong enough alternatives that we can almost certainly reject the null. 
\Cref{assumption:continuity} is a technical assumption 
assuming that the probability that the p-value is exactly $0$ is zero  
under any alternative. The completeness in 
\Cref{assumption:complete} is to guarantee that 
if two level $\alpha$ meta-analysis tests for $H_{0u}$ 
has the same power at every point in the 
alternative space then they are the same test. 
 Roughly speaking, both \Cref{assumption:power,assumption:complete}
 require that the 
 alternative space of each individual hypothesis is large enough to include 
 various possibilities. The three assumptions can be satisfied in common tests.
 
 For example, tests satisfying \Cref{assumption:power}  
 include testing the parameters of exponential families and location families. 
 \citet[Theorem 4.3.1]{lehmann2006testing} show that completeness is satisfied
 for testing the natural parameter of a $k$-dimensional
  exponential family if the alternative space $\Theta_{1 i}$ contains a
  $k$-dimensional rectangle. If the individual p-values are independent, 
  then completeness of the alternative space for each individual hypothesis 
  implies \Cref{assumption:complete}.
  Thus, we believe that 
\cref{assumption:power,assumption:complete,assumption:continuity}
  can cover a large class of problems.  

Theorem~\ref{thm:necessity} shows that GBHPC p-values form a complete class of  
monotone $\alpha$-admissible
p-values for $H_0^{r/n}$. 
Theorem~\ref{thm:sufficiency} states that a sufficient  
condition for a sensitive GBHPC p-value to be monotone $\alpha$-admissible is 
that each $g_u$ is admissible for $H_{0u}$.

\begin{thm}\label{thm:necessity}
Let $P_1,\dots,P_n$ be independent or positively dependent (PRDS) 
P-values satisfying
\cref{assumption:power,assumption:continuity}.
Let $P_{r/n}$ be a valid monotone p-value for $H_0^{r/n}$.
Then there exists a valid GBHPC p-value
$p_{r/n}^\star$ 
    that is uniformly at least as powerful as $P_{r/n}$. 
\end{thm}

\begin{thm}\label{thm:sufficiency}
Let $P_1,\dots,P_n$ be PRDS 
P-values satisfying 
\cref{assumption:power,assumption:continuity,assumption:complete}.
For a sensitive GBHPC p-value $P_{r/n}^\star=f^\star(\bsP)$ of the form~\eqref{eq:defgbhpc}, a sufficient condition for $P_{r/n}^\star$ to 
be monotone $\alpha$-admissible is that each $g_u$ 
is an admissible meta-analysis p-value for $H_{0u}$.
\end{thm}

We introduce the following lemma, which is the key reason 
that \Cref{thm:necessity} and \Cref{thm:sufficiency} hold. 
It shows that given a valid monotone p-value that is not
of the GBHPC form, we can expand its
rejection region while retaining its validity.

\begin{lem}\label{lem:construct}
Let $P_1,\dots,P_n$ be PRDS P-values satisfying
\cref{assumption:power}. 
Let $f(P_1, \cdots, P_n)$ be a valid monotone p-value for
  $H^{r / n}_0$ and for $u\subset1{:}n$ with $|u|=n-r+1$, 
define
  \begin{equation}\label{eq:define_g}
    g_{u} (\bsP_u) = \inf_{\bsP_{-u} \in (0, 1]^{r-1}} f (P_1, \cdots, P_n)
  \end{equation}
   Then $g_u$ is a valid monotone meta-analysis p-value for $H_{0u}$.
\end{lem}

\begin{proof}
Monotonicity of $f$ implies monotonicity and measurability of  $g_{u}$.
Next, suppose that $g_u$ is not valid for $H_{0u}$. Then there is an
$\alpha\in[0,1]$ and a $\theta_u^\star$ with $\theta_j\in\Theta_{0j}$ for all $j\in u$
such that 
$\Pr_{\theta_u^\star}( g_u(\bsP_u)\le\alpha) = \Pr_{\theta_u^\star}\big(\inf_{\bsP_{-u} \in (0, 1]^{r-1}} f (\bsP) \le \alpha
\big)>\alpha+\epsilon$ for some $\epsilon>0$.
From the monotonity of $f$, there is some fixed $\wt p\in(0,1]$ with
$\Pr_{\theta_u^\star}(f(\bsP_u{:}{\bsp}_{-u})\le\alpha)>\alpha+\epsilon$ for any 
$\bsp_{-u} \in [0, \wt p]^{r-1}$. Since the p-values are PRDS and 
$\{\bsp: f(\bsp_u, \bsp_{-u}) \leq \alpha\}$ 
is a decreasing set for any fixed $\bsp_{-u}$, we have  
\[
\Pr_\theta\left(f(\bsP)\le\alpha\mid \bsP_{-u} \in [0, \wt p]^{r-1}\right) 
\geq \alpha + \epsilon/2
\]
for any $\theta$ satisfying $\theta_u = \theta_u^\star$. 
 Using  \Cref{assumption:power}, there also exists
$\theta_{-u}^\star$ with $\theta_j^\star\in\Theta_{1j}$ for $\forall j \in -u$ such that
$\Pr_{\theta_j^\star}(P_j \leq \wt p)\geq 
\big((\alpha + \epsilon/2)/(\alpha + \epsilon)\big)^{1/(r-1)}$. Since 
the p-values are PRDS, using \eqref{eq:prds} we have 
$$\Pr_{(\theta_u^\star{:}\theta^\star_{-u})}\big(f(\bsP)\leq \alpha\big) \geq
\Pr_{(\theta_u^\star{:}\theta^\star_{-u})}
\big(f(\bsP)\leq \alpha, P_{j}\leq\wt p, \forall j\in -u\Big)> \alpha + \epsilon/2$$
contradicting the validity of $f(\bsP)$.
\end{proof}

Now we are ready to prove Theorems~\ref{thm:necessity} and~\ref{thm:sufficiency}.

\begin{proof}[\bf Proof of \Cref{thm:necessity}]
Let $g_u(\bsP_u)$ be defined in \eqref{eq:define_g}. Then $P_{r/n} \ge P_{r/n}^\star$ when
$\bsP \in (0, 1]^n$. Using \Cref{assumption:continuity}, $P_{r/n}^\star$ is then uniformly 
at least as powerful as $P_{r/n}$.
It then follows directly from Lemmas~\ref{prop:validity} and~\ref{lem:construct} that 
$p_{r/n}^\star$ is a valid GBHPC p-value. 
\end{proof}

Using Lemma~\ref{prop:validity}, to prove \Cref{thm:sufficiency}, 
we only need to prove the monotone $\alpha$-admissibility of $P_{r/n}^\star$.

\begin{proof}[\bf Proof of \Cref{thm:sufficiency}]
To prove the monotone $\alpha$-admissibility
of $f^{\star}(P_1, \cdots, P_n)$, suppose that there is a valid monotone test $f^{\star\star}$ satisfying 
$\Pr_\theta( f^{\star\star}(P)\le\alpha)\ge \Pr_\theta( f^{\star}(P)\le\alpha)$
for all $\theta\in\Theta_1^{r/n}$.
By \Cref{thm:necessity} we can assume that $f^{\star\star}$ is a GBHPC p-value:
$$f^{\star \star} (\bsP) = \max_{u\subset1{:}n\atop|u|=n-r+1} g_{u}' (\bsP_u),$$
where $g'_u$ is a valid monotone meta-analysis p-value. Notice that 
since $f^\star$ is sensitive, equation \eqref{eq:unique_g} holds. 
We now show that for each $u\subset1{:}n$ with $|u|=n-r+1$, and any $\theta_{u}\not\in\Theta_{0u}$,
    \begin{equation}\label{eq:g_adm}
      \Pr_{\theta_{u}} \Big(\inf_{\bsP_{-u} \in (0, 1]^{r-1}}f^\star(\bsP) \leq \alpha\Big) \leq 
	  \Pr_{\theta_{u}} (g_{u}'(\bsP_u) \leq \alpha) \equiv\beta'
    \end{equation}
 using a similar strategy as in the proof of \Cref{lem:construct}. 
	If \eqref{eq:g_adm} does not hold for some set  $u$
and a corresponding $\theta_{u}$, then
there exist some $\epsilon > 0$ and $\wt p \in (0, 1]$ such that
$\Pr_{\theta_{u}} (f^{\star} (\bsP_{u}{:}\bsp_{-u}) \leq \alpha) > \beta' + \epsilon$
 for any $\bsp_{-u}\in(0, \wt p]^{r-1}$. Using \Cref{assumption:power}, there exists $\theta^\star$ with
 $\theta_j^\star \in \Theta_{1j}$ for $j \in -u$ such that     
 $ \Pr_{\theta_j^\star} (P_j \leq \wt p) \geq \bigl( (\beta' +
       \epsilon / 2)/(\beta' + \epsilon)\bigr)^{1/(r - 1)} $. 
    Thus,
	\begin{align*}
         \Pr_{(\theta_u{:}\theta^{\star}_{-u})} (f^{\star} (\bsP) \leq \alpha)
      &\geq  \Pr_{(\theta_u{:}\theta^{\star}_{-u})} (f^{\star} (\bsP)
      \leq \alpha, P_j \leq \wt p, \forall j \in -u)
       >   \beta' + \epsilon/2\\
      &>  \Pr_{\theta_u} (g_u'(\bsP_u) \leq \alpha)
       \geq  \Pr_{(\theta_u{:}\theta^{\star}_{-u})} (f^{\star \star} (\bsP) \leq \alpha)
    \end{align*}
    which violates the assumption that $f^{\star\star}$ is uniformly at least as powerful as $f^\star$. Thus, \eqref{eq:g_adm} holds.
 equation \eqref{eq:unique_g} and \eqref{eq:g_adm} implies that
 $\Pr_{\theta_{u}} (g_{u}' \leq \alpha) \geq  \Pr_{\theta_{u}} (g_{u}(\bsP_u) \leq \alpha)$ for any 
 $\theta_u \not\in\Theta_{0u}$ and any $\alpha \in [0, 1]$. 
    As $g_{u}(\bsP_u)$ is ${\alpha}$-admissible for $H_{0u}$, we have
    $ \Pr_{\theta_{u}} (g_{u}' (\bsP_u) \leq \alpha) = 
    \Pr_{\theta_{u}} (g_{u}(\bsP_u) \leq \alpha) $.
    Further, using \Cref{assumption:complete} we have 
    $ g_{u}' (\bsP_u) = g_{u}(\bsP_u)\ \mathrm{a.e.}$.
    Thus, for all $\theta \in
       \Theta_1^{r/n}$,
    $\Pr_{\theta} (f^{\star \star} (\bsP) \leq
       \alpha) = \Pr_{\theta} (f^{\star} (\bsP) \leq
       \alpha)$
    which shows that $f^{\star}$ is monotone ${\alpha}$-admissible for
    $H^{r / n}_0$.
\end{proof}

\begin{rem}
As mentioned in \Cref{rem:bhpc}, 
the BHPC p-values are symmetric GBHPC p-values 
As a consequence, the BHPC p-values characterize the form of 
symmetric monotone admissible combined p-values.
\end{rem}
%

Combining \Cref{thm:sufficiency} with results of 
\cite{birnbaum1955characterizations} and
\citet[Theorem 6.7.1]{lehmann2006testing} 
who characterized admissible tests for the global null
in exponential families, 
we can  give more specific conditions 
when \Cref{thm:sufficiency} is applied to exponential families. 
here is the result for a simple senario.

\begin{exmp}
  Suppose that independent test statistics $T_i$ for $i=1,\dots,n$
  are available on hypotheses $H_{0i}:\bstheta_i = \bm{a}_i \in \mathbb{R}^{k_i}$ 
  against $H_{0i}:\bstheta_i \in \mathbb{R}^{k_i}/\{\bm{a}_i\}$. 
  Here we assume that every $T_i$ is the sufficient statistics 
  for an exponential family with 
  natural parameter $\bstheta_i$.  
  For a sensitive GBHPC p-value $f^\star(\bsp)$,  
  suppose that $\Pr_{\bstheta_{0u}}(g_u(\bsP_u) \leq \alpha) = \alpha$.
  Also, for $\forall \alpha \in [0, 1]$ the 
  set of $\bm T_u$ for which $g_u(\bsp_u)>\alpha$ (the acceptance region) 
  is a closed and convex set, except for a subset of measure $0$.
  Then $f^\star(\bsp)$ is 
  monotone $\alpha$-admissible for $H_0^{r/n}$.
\end{exmp}


Related work on convexity and admissibility also appears in
\cite{matthes1967tests} for testing parameters of exponential families with presence of nuisance parameters, 
\cite{marden1982} and \cite{brown1989} for generalization to distribution families beyond exponential families, and
\cite{owen2009} for tests powerful against alternatives with concordant signs. 
Notice that the $n$-dimensional set of test statistics $\bm T$ 
itself for which $f^\star(\bsp) > \alpha$ is not convex. For partial conjunctions, the null hypothesis for the 
parameter usually includes all of the coordinate axes and the smallest convex set containing
the axes is all of Euclidean space. As a result convexity of the acceptance region is not appropriate
to partial conjunction testing.


%

\subsection{Inadmissibility}\label{sec:inadmiss}

In \Cref{sec:admissibility}  
we  constructed monotone $\alpha$-admissible p-values for $H^{r/n}_0$, we 
we show that they  
fail to be admissible if we allow non-monotone tests. For the case $n = r = 2$, the construction of such counter-examples dates back to
\cite{lehmann1952} and \cite{iwasa1991}.

Here we demonstrate that if we don't require monotone tests then
a BPHC test is inadmissible.
  Let $n = r = 2$.
  If both $P_1$ and $P_2$ are $\alpha$-admissible,
  then using \Cref{thm:necessity} and \Cref{thm:sufficiency}, the constructed combined
  p-value is just $P_{(2)}$, which is monotone admissible. At a given $\alpha$,
  the critical function is
  $ \varphi = \mathbbm{1}_{p_{(2)} \leq \alpha }(p_1, p_2)$.

  Now we can easily construct a more powerful $\alpha$-level test, by adding to
  the original rejection region a square around the top-right corner in the p-value 
  space (Solid shaded regions in \Cref{figure:counter_example}). Define the set
      \[ S = 
\begin{cases}
          \{ (p_1, p_2) \mid p_{(1)} \geq 1 - \alpha \},  & \text{ if } \alpha < \frac{1}{2}\\
         \{ (p_1, p_2) \mid p_{(1)} \geq \alpha \},        & \text{ if } \alpha \geq \frac{1}{2}.
\end{cases}
\]
     Then the test $\varphi'$ with critical function 
     $\varphi'(\bsP) = \varphi(\bsP) + \mathbbm{1}_{(P_1, P_2)\in S}$ 
     is uniformly and strictly more powerful than $\varphi$. To prove that $\varphi'$ is 
     an $\alpha$-level test, we note that $S \cap \{p_{(2)}\leq \alpha\} = \emptyset$.
     Therefore $\e(\varphi'(\bsP)\mid P_1 = p_0) \leq \alpha$  holds for any $p_0\in[0,1]$.
     Similarly, $\e(\varphi'(\bsP)\mid P_2 = p_0)\leq \alpha$. Since $p_0$ is arbitrary we
    conclude that $\varphi'$ is an $\alpha$-level test. 
    Actually, as shown in \Cref{figure:counter_example}, 
    we can further expand the rejection region of $\varphi'$ to 
    include also the dotted shaded regions and to get an even more powerful but 
    still valid test   
    $\widetilde\varphi$. The rejection region of $\widetilde\varphi$ 
    in the p-values space consists of small squares along the diagonal line. 
    
     If the test statistics are $Z_1 \sim \dnorm(\mu_1, 1)$ and 
     $Z_2 \sim \dnorm(\mu_2, 1)$, and $H_{1}$ and $H_2$ are two-sided tests for the mean
     $\mu_1$ and $\mu_2$ respectively, then the top two plots of 
     \Cref{figure:counter_example} 
     show the rejection region of $\varphi'$ and $\widetilde\varphi$ at level 
     $\alpha = 0.1$ 
     in the p-value space and in the test statistic space. 
     The bottom two plots compare the power of $\varphi$ and $\widetilde \varphi$ 
     as a function of $(\mu_1, \mu_2)$. They show that the power gain of 
     the non-monotone $\widetilde\varphi$
     only appears in the low power region where the power is below or near $\alpha$.

     \begin{figure}[ht]
\includegraphics[width=\textwidth]{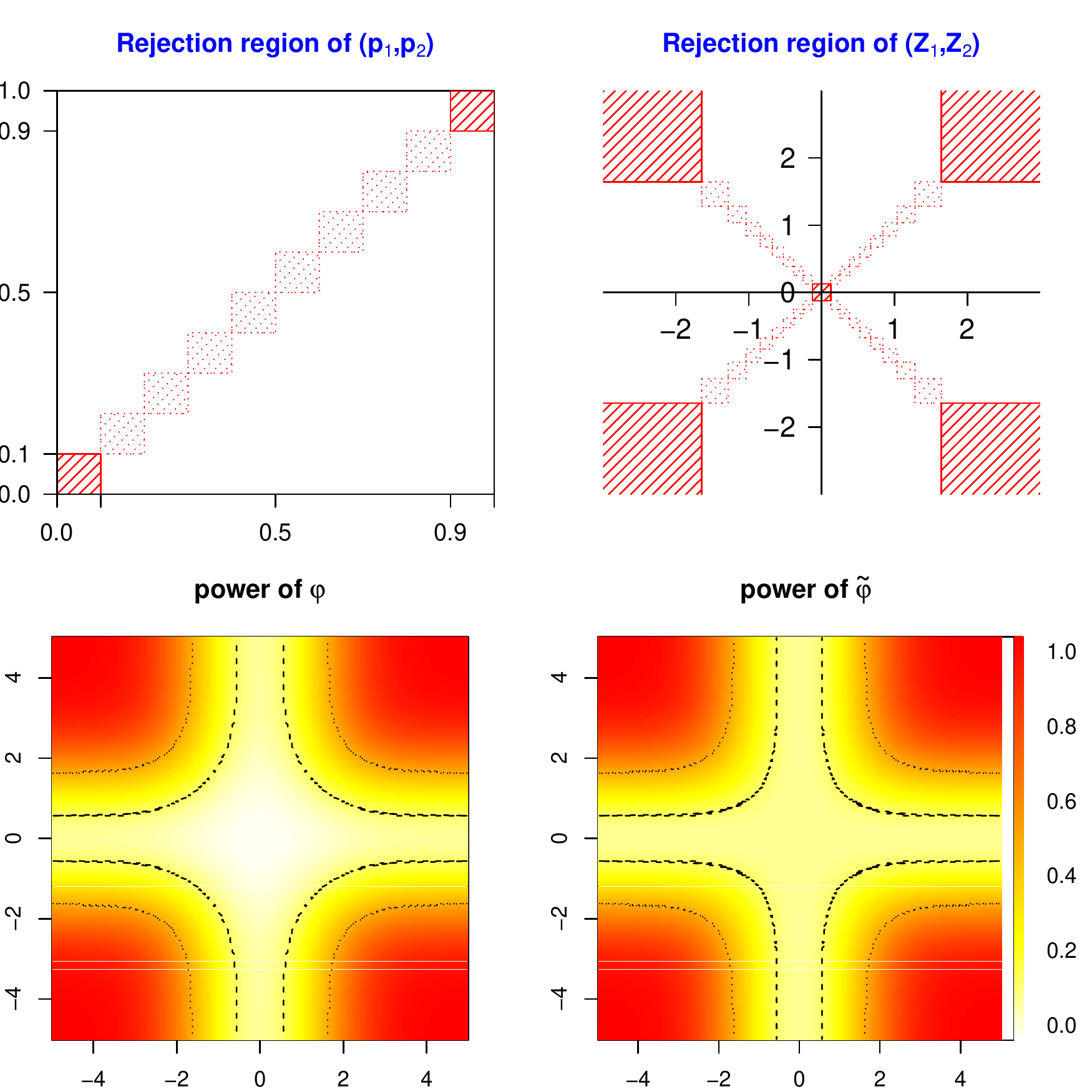}
\caption{The top two plots: rejection regions of $\varphi'$ and 
$\widetilde\varphi$ in the p-value space and the test statistic space, using $\alpha = 0.2$. The sold shaded region is the rejection region of $\varphi'$, while 
the rejection region of $\widetilde \varphi$ also includes in the dotted shaded squares. The Bottom two plots: power comparison of $\varphi$ and $\widetilde \varphi$. The 
dashed line is where power is at $0.15$ and the dotted line is where the 
power is $0.5$.}
\label{figure:counter_example}
\end{figure}

The more powerful test $\varphi'$
increases power by strangely rejecting $H^{2/2}_0$ when both input p-values are large enough.
We now use this same approach to
show that without the monotonicity constraint, 
any GBHPC p-value is inadmissible 
for any $n$ and any $r\in 2{:}n$. 
The counter-examples reject $H^{r/n}_0$ when all p-values are large.
The idea is to show that for any GBHPC test, it's always possible to add 
a ``box''-shaped rejection region like the square around the origin 
in the right panel of 
\Cref{figure:counter_example} while still keeping the test valid. 
The point is not to advocate for such tests, but rather to reinforce the idea
that admissibility is only a useful concept within a well chosen class of functions.

We need the following mild technical constraint to guarantee that the 
``box'' we choose can really increase power at least in one alternative hypothesis.
\begin{assumption}\label{assumption:dominate}
  For each $i \in 1{:}n$, there exists 
  $\theta_i^0 \in \Theta_{0i}$ that 
  $\Pr_{\theta_i^0}(P_i \leq \alpha) = \sup_{\theta_i \in \Theta_{0i}}(P_i \leq \alpha)$ 
  for $\forall \alpha \in [0, 1]$. 
  Let $\theta^0 = (\theta_1^0, \theta_2^0, \cdots, \theta_n^0)$.  Then 
  for any set $A$, if $\Pr_{\theta^0}
    (A) > 0$, then there exists $ \theta^1 \in \Theta_1^{r /
    n}$  that
    $      \Pr_{\theta^1} (A) > 0 $.
\end{assumption}

\begin{thm}
Let $P_1,\dots,P_n$ be independent p-values satisfying
\cref{assumption:power,assumption:continuity,assumption:dominate}.
Let $1<r\le n$ and $\alpha \in (0, 1)$.  Then any monotone $\alpha$-admissible
combined p-value for testing $H^{r / n}_0$ is not $\alpha$-admissible
without the monotonicity constraint.
\end{thm}

\begin{proof}
  Using \Cref{thm:necessity}, we only need to consider a GBHPC p-value 
  $f^\star$ which is defined in
  \Cref{def:GBHPC}. Let $\theta^0 = (\theta_1^0, \theta_2^0, \cdots, \theta_n^0)$ be the 
  parameter in \Cref{assumption:dominate}.   
define  
$$R = \{\bsp \in [0, 1]^n:f^\star(\bsp)\leq \alpha\} 
= \bigcap_{u\subset1{:}n\atop |u|=n-r+1} R_u$$
  where $R_u = \{\bsp \in [0, 1]^n:g_{u}(\bsp_u) \leq \alpha\}$ 
  and $g_{u}$ is defined in~\eqref{eq:defgbhpc}.

  First, as $\Pr_{\theta^0}(f^\star \leq \alpha) \leq \alpha < 1$ 
  and $f^\star$ is non-decreasing,  
  there exists some $p_0 < 1$ such that if $p_{j} \geq p_0$ for all $j \in 1{:}n$ then
$ f^{\star} (\bsp) > \alpha$. 

  Then, we show that there must exist a set $u^\star$ with
  $\Pr_{\theta^0}(R_{u^\star} \cap R^c) = \epsilon > 0$, where $R^c$ is the complement set of $R$. 
  If this doesn't hold, then it means that 
  for any $u \subset 1{:}n$ with $|u| = n-r+1$, the equation 
  $1_{f^\star(\bsp) \leq \alpha}(\bsp) = 1_{g_u(\bsp_u)\leq \alpha}(\bsp)~a.e.~\Pr_{\theta^0}$
  holds. This implies that $1_{f^\star \leq \alpha}$ doesn't depend on $\bsp_{-u}$ 
  except for a zero probability set under $\Pr_{\theta^0}$. As 
  $\cup_{u\subset1{:}n\atop |u|=n-r+1} {-}u = 1{:}n$, we get that  
  $1_{f^\star \leq \alpha}$ doesn't depend on any $p_j$ except for a zero probability set 
  under $\Pr_{\theta^0}$, 
  which implies that $1_{f^\star \leq \alpha} \equiv 1$ or $0~ a.e.~\Pr_{\theta^0}$. 
  It's obvious that 
  such a test is either invalid or trivially not admissible, which contradicts our 
  assumptions.



  As a consequence, 
  we have 
  $\Pr_{\theta^0}(f^\star\leq \alpha) = \Pr_{\theta_0}(R_{u^\star}) - \epsilon
  \leq \alpha - \epsilon$.
   Notice that $\Pr_{\theta^0}(f^\star\leq \alpha) 
  = \mathbb{E}_{\theta^0_{-u}}\big(\Pr_{\theta^0_u}
  [f^\star\leq \alpha\mid \bsP_{-u}]  \big)$ 
  for any $u$. 
 using the fact that $f^{\star}$ is non-decreasing, 
 $\Pr_{\theta^0_{u}} [ f^{\star} \leq \alpha \mid \bsP_{-u} = \bsp_{-u}
 ]$ is non-increasing in $\bsp_{-u}$. Thus 
 there exists  $\wt p < 1$, such that for any $u$, if $\bsp_{-u} \in [\wt p, 1]^{r - 1}$, then
\[ \Pr_{\theta^0_{u}} [ f^{\star} \leq \alpha \mid \bsP_{-u} = \bsp_{-u}
] \leq \alpha - \epsilon. \]
  Let $p^{\star} = \max (p_0, \wt p, 1 - \epsilon^{1 / (n - r + 1)})$ and 
  $S = \cap_i \{\bsp \in [0, 1]^n: p_i \geq p^\star\}$.
  Then we construct a new test with critical function $\varphi$:
  $\varphi = \mathbbm{1}_{f^{\star} \leq \alpha} + \mathbbm{1}_S$.
 
  As $\{\bsp\in[0, 1]^n: f^{\star}(\bsp) \leq \alpha\} \cap S = \emptyset$, we know  
  that $\varphi$  is at least as powerful as $\mathbbm{1}_{f^{\star} \leq \alpha}$.
  Using \Cref{assumption:dominate}, as $\Pr_{\theta^0}(S) \geq (1 - p^{\star})^n > 0$, 
  there exists $ \theta^1 \in \Theta_1^{r/n}$ with $\Pr_{\theta^1}(S) > 0$. Thus, 
  $\varphi$ strictly dominates $\mathbbm{1}_{f^{\star} \leq \alpha}$ at $\theta^1$.
    Finally, for $\forall \bsp \in [0, 1]^n$ and $\forall u\subset 1{:}n$ with $|u| = n-r+1$,
    if $\theta_u \in H^{1/{n -r + 1}}$, then
    \begin{align*}
    \mathbbm{E}_{\theta_{u}} [ \varphi \mid \bsP_{-u} = \bsp_{-u} ]
    &\leq \Pr_{\theta_{u}} [ f^{\star}\leq \alpha \mid \bsP_{-u} = \bsp_{-u} ] 
    + \epsilon \mathbbm{1}_{\bsp_{-u} \in [p^{\star}, 1]^{r-1}}\\
    & \leq \Pr_{\theta_{u}^0} [ f^{\star}\leq \alpha \mid \bsP_{-u} = \bsp_{-u} ] 
    + \epsilon \mathbbm{1}_{\bsp_{-u} \in [p^{\star}, 1]^{r-1}}
    \leq \alpha.
  \end{align*}
  The second inequality above follows from \Cref{assumption:dominate}, independence of 
  the individual p-values and monotonicity of $f^\star$. 
  Thus $\varphi$ is still an $\alpha$-level test for $H^{r/n}$. This shows that 
  $f^{\star}$ is not
  $\alpha$-admissible.
\end{proof}

\section{Simulation}\label{sec:simulation}
In this simulation example, we compare the power of several GBHPC p-values testing 
for the PC hypothesis $H_0^{r/n}$ with $n = 8$ studies and $r = 2$. Compared with other 
$r$ values, the null hypothesis $H_{0}^{2/n}$ is often of 
particular interest as it tests whether the significance of the effect can 
replicate or not across studies. It is also the case where the computational cost of 
non-symmetric GBHPC p-value would typically not be a concern.

We consider the alternative whose true number of 
non-null hypotheses is one of $r_0 = 2, 4, 6$. We assume that all the individual 
hypotheses are independent. Each p-value $P_i$ is 
a two-sided p-value of the corresponding z-value $Z_i \sim \dnorm(\sqrt{N_i}\mu_i, 1)$ 
for $i = 1, 2, \dots, 8$. 
We set three of the sample sizes $N_i$ of the eight individual studies to $100$, another three of them to $500$ and the last two to $1000$.  
for the effects $\mu_i$, if $H_{0i}$ is true, then $\mu_i = 0$. 
Otherwise when $H_{0i}$ is false, we generate 
$\mu_i \overset{i.i.d}{\sim} \text{Gamma}(\alpha_0, \beta_0)$. We define 
$\mu_0 = \alpha_0/\beta_0$ and $\sigma_0 = \sqrt{\alpha_0/\beta_0^2}$ which are the mean
and standard deviation of the non-null effect across studies. 
We compare the power of each GBHPC p-value as a function of $(\mu_0, \sigma_0)$ 
at the significance level of $\alpha = 0.05$.

We compare three GBHPC P-values, whose meta-analysis P-values $g_u$ are from 
\cref{example:simes,example:fisher,exmp:wtstouf} respectively. 
The results are shown in \Cref{figure:simulation}. 
For each $r_0$ and each form of the GBHPC p-value, 
we plot a power map against $(\mu_0, \sigma_0)$. 
To better illustrate the difference across methods, for Simes and weighted Stouffer 
GBHPC p-values, we plot their powers after they are substracted by the powers of Fisher's BHPC p-value at the corresponding location. Here are some 
observations from the simulation results. First, \Cref{figure:simulation} shows that when $r_0 = 2$, Simes BHPC p-value 
is the most powerful in a large region of the alternative space. The 
reason is that, for each subset of hypotheses with size $n -r + 1 = n - 1$, 
at the worst case there is only one non-null individual hypothesis, where 
simes should be most powerful in detecting extreme p-values. 
Second, the Weighted Stouffer 
GBHPC p-value can have higher power than Fisher's BHPC p-value when $r_0 > r = 2$ 
and when the effect heterogeneity across non-null studies ($\sigma_0$) is not too large 
 to dominate the average effect ($\mu_0$).  
Notice that we are using two-sided p-values to make a fair comparison of the methods, 
thus taking $\sqrt{n_i}$ as 
weights in Stouffer's method 
would not be optimal. However, as we have discussed in \Cref{exmp:wtstouf} and shown 
here, it can still provide a 
good test with two-sided p-value. 
If individual p-values were one-sided, then we would have seen a even higher power gain.
Finally, 
the three methods would not have a noticeable difference in power when the power is 
too low or too high. Most of the difference apprear when the power is in the range of 
$0.4$ to $0.6$.

\begin{figure}[ht]
\includegraphics[width=\textwidth]{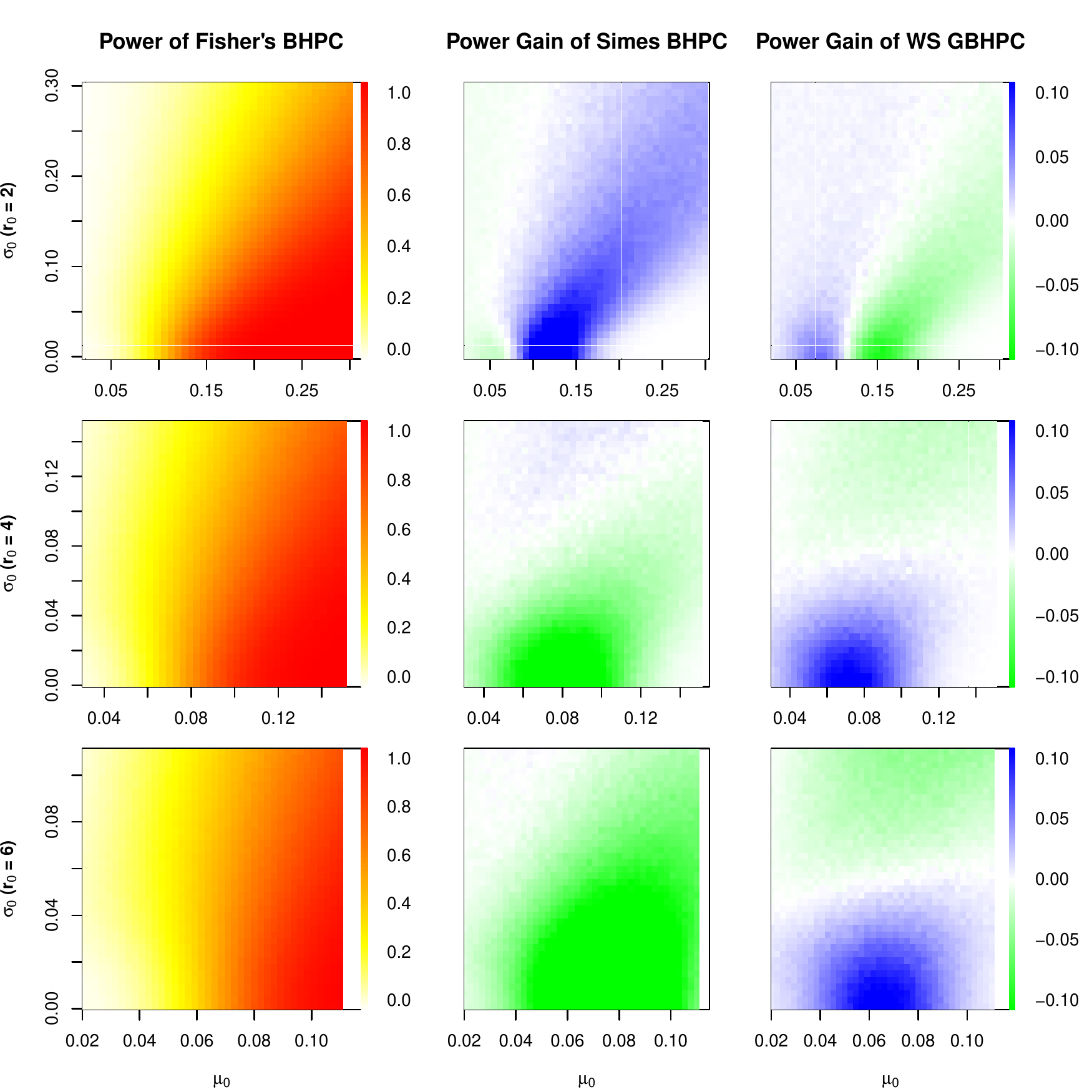}
\caption{Power comparison of GBHPC p-values: 
Each row is for one $r_0$ value and each column is for one form of GBHPC p-value. 
the significance level is $\alpha = 0.05$. 
The first column are power maps Fisher's BHPC p-value against $(\mu_0, \sigma_0)$. 
the last two columns are the power difference of Simes and Weighted Stouffer's GBHPC 
p-values against Fisher's p-value. The green color indicates a power loss compared 
with Fisher's BHPC p-value while the blue color shows a power gain.}
\label{figure:simulation}
\end{figure}

\section{Real data analysis}\label{sec:real}

Compared with BHPC p-values, GBHPC p-values has the flexiblility to 
make use of the possibly complex dependency structure across studies, thus 
may achieve higher power. We use a real data example to illustrate this benefit.

The dataset \citep{ruff2014} is a pooled dataset from four randomized 
clinical trials aiming to measure the relative benefit of new anticoagulants (NOAC) 
compared with an old drug called warfarin for stroke prevention. One primary goal in 
the original paper was to assess and compare the efficiency of these 
new drugs in different clinical subgroups of patients. The subgroups and the data  
are shown in \Cref{table:1}. The data ($m/N$) records that 
among $N$ samples the number of samples who suffer a stroke or systemic embolic event 
is $m$. 
Thus, to test the different between using the old and new drug, we 
estimate the odds ratio and compute p-values for each subgroup using Fisher's exact test. 

We can use PC tests to assess the consistency of the drug efficiency 
across different subgroups, which is one of the major interest of the original paper. 
One major difficulty for applying PC tests (and meta-analysis) to theses subgroups 
is that these subgroups are correlated. For example, the subgroup of age $\leq 75$ 
clearly overlaps with the subgroup of female. When there is unknown dependency across 
individual hypotheses, the only BHPC p-values available are to have the 
meta-analysis p-value $g$ be either the Bonferroni or Simes p-value. 
The left column of \Cref{table:2} are the values of the Bonferroni BHPC p-values 
$p_{r/n}^{\text{Bon}}$
for $H_{0}^{r/n}$ with $n = 18$ and $r$ changing from $2$ to $18$. 
As $p_{r/n}^{\text{Bon}}$ is increasing in $r$ except when $r$ change from $16$ to 
$17$, the values of Simes BHPC p-values will agree with $p_{r/n}^{\text{Bon}}$ 
except when $r = 16$. 

The question is: is it possible to get smaller valid p-values? The answer is YES, 
by using a non-symmetric GBHPC p-value. Notice that the groups which 
represent different levels of the same grouping factor do not share samples, thus have 
independent p-values. For example, the three p-values for the three $\text{CHADS}_2$ 
groups are independent. 
Let $1{:}18 = \cup_{i = 1}^8 I_{i}$ where each $I_{i}$ is the index set of subgroups 
using the $i$th grouping factor. 
Then, to build a GBHPC p-value $p_{r/n}^{\text{new}}$, 
we construct $g_u(p_u)$ for each $u\in 1{:}n$ with $|u| = n - r +1$ 
as follows: for each $I_i\cap u \neq \emptyset$, a Fisher's p-value (\Cref{example:fisher}) 
$p_{u, i}$ is calculated on $p_{u \cap I_i}$, then 
$$g_u(p_u) = \big|\{i: u \cap I_i \neq \emptyset\}\big|\cdot\left(\min_i p_{u, i}\right)$$
which is a Bonferroni combination of p-values 
across the grouping factors. The above construction 
obviously provides a valid GBHPC p-value.

The next concern is the 
computation of $p_{r/n}^{\text{new}}$. We claim that for any given 
$r$, it can be quickly computed by checking \Cref{table:3}. 
\Cref{table:3} are all $p_{u, i}$ values 
that can possibly affect the value of $p_{r/n}^{\text{new}}$ for any $r$. 
Notice that 
$$f^\star(p) = \max_{u, |u| = n - r + 1}\min_i p_{u, i}.$$
Since $p_{u, i}$ is symmetric, it can possibly influence $f^\star$ only when 
$p_{u, i}$ is the Fisher's combination on the largest $|\mu \cap I_i|$ p-values in 
$I_i$. This explains the values in \Cref{table:3}. Then, computing 
$p_{r/n}^{\text{new}}$ becomes easy. One nice phenomenon of the p-values in 
\Cref{table:3} is that p-values on the first row are uniformly smaller 
than the p-values on the second row, the latter further uniformly smaller than 
p-values on the third row.
Such a property can 
greatly simplify the computation. The calculation of $g_u(p_u)$ 
is by replacing some $p_{I_i, i}$ (first row p-value in \Cref{table:3}) with 
some larger $p_{u, i}$ (higher row p-values) or completely removing it. 
For example when $r = 4$, there are $n -r + 1 = 3$ indices that are not in $u$. 
Thus for any $u$, at most $3$ of the $p_{I_i, i}$ can be replaced. If we 
denote $p_{I_{(1)}, (1)} \leq \dots p_{I_{(8)}, (8)}$, then it's easy to see that 
$p_{4/n}^{\text{new}} = 8p_{I_{(4)}, (4)}$. 

It is a bit more complicated when $r \geq 9$. First, notice that for the 
$r - 1$ indices that are not in $\mathrm{argmax}_u g_u(p_u)$, we should 
have each $I_i \cap -u \neq \emptyset$ to avoid appearance of p-values on the first 
row in \Cref{table:3}. 
Then the problem becomes optimizing the location of the rest $r - 1- 8$ indices 
to maximize $g_u(p_u)$. For example when $r = 12$, there are $4$ indices left and 
we now just need to examine the p-values on the second and third row in \Cref{table:3}. If we denote the p-values on the second row as $p_{u_i, i}$ and 
define $p_{u_{(1)}, (1)} \leq \dots p_{u_{(8)}, (8)}$, using a similar 
argument as $r = 4$, one can check that 
$p_{12/n}^{\text{new}} = 
\max\{7p_{u_{(2)}, (2)}, 6p_{u_{(3)}, (3)}, 5p_{u_{(4)}, (4)}\}$. 

Finally,  the $p_{r/n}^{\text{new}}$ values for all $r = 2, \dots, 18$ are shown 
in \Cref{table:2}. Compared with Bonferroni BHPC p-values, the new GBHPC p-values 
can be much smaller especially when $r$ is small. Both methods give a $95\%$ confidence interval of the true proportion of non-null hypotheses $r_0/n$ as $ r_0/n \in [0.72, 1]$.

\begin{table}[ht]
\begin{subtable}{.6\textwidth}
\caption{}
\label{table:1}
\centering
\scalebox{0.6}{
\begin{tabular}{@{}lrrrr@{}}
  \hline
   & \multirow{2}{2.6cm}{Pooled NOAC (events)}& 
   \multirow{2}{3cm}{Pooled Warfarin (events)} & 
   \multirow{2}{2.1cm}{Estimated Odds Ratio} & 
   \multirow{2}{1.3cm}{p value} \\ 
   &&&&\\
  \hline
 \multicolumn{3}{@{}l}{\textbf{Age(years)}}\\ 
$\leq 75$ & 496/18073 & 578/18004 & 0.85 & 9.26E-03 \\ 
  $\geq 75$ & 415/11188 & 532/11095 & 0.76 & 6.61E-05 \\ 
   \multicolumn{3}{@{}l}{\textbf{Sex}}\\ 
  Female & 382/10941 & 478/10839 & 0.78 & 5.00E-04 \\ 
  Male & 531/18371 & 634/18390 & 0.83 & 2.38E-03 \\ 
   \multicolumn{3}{@{}l}{\textbf{Diabetes}}\\ 
  No & 622/20216 & 755/20238 & 0.82 & 2.93E-04 \\ 
  Yes & 287/9096 & 356/8990 & 0.79 & 3.81E-03 \\ 
   \multicolumn{3}{@{}l}{\textbf{Previous stroke or TIA}}\\ 
  No & 483/20699 & 615/20637 & 0.78 & 4.65E-05 \\ 
  Yes & 428/8663 & 495/8635 & 0.85 & 2.14E-02 \\ 
 \multicolumn{3}{@{}l}{\textbf{Creatinine clearance (mL/min)}}\\ 
  $\leq 50$ & 249/5539 & 311/5503 & 0.79 & 6.24E-03 \\ 
  $50-80$ & 405/13055 & 546/13155 & 0.74 & 5.85E-06 \\ 
  $\geq 80$ & 256/10626 & 255/10533 & 1.00 & {\color{blue}9.64E-01} \\ 
  \multicolumn{3}{@{}l}{\textbf{$\text{CHADS}_2$ score}}\\ 
 $0-1$ & 69/5058 & 90/4942 & 0.75 & {\color{blue}7.83E-02} \\ 
  $2$ & 247/9563 & 290/9757 & 0.87 & {\color{blue}1.05E-01} \\ 
  $3-6$ & 596/14690 & 733/14528 & 0.80 & 5.21E-05 \\ 
  \multicolumn{3}{@{}l}{\textbf{VKA status}}\\  
  Naive & 386/13789 & 513/13834 & 0.75 & 2.19E-05 \\ 
  Experienced & 522/15514 & 597/15395 & 0.86 & 1.61E-02 \\ 
   \multicolumn{3}{@{}l}{\textbf{Centre-based TTR}}\\ 
  $\leq 66$ & 509/16219 & 653/16297 & 0.78 & 2.49E-05 \\ 
  $\geq 66$ & 313/12742 & 392/12904 & 0.80 & 4.68E-03 \\ 
   \hline
\end{tabular}
}
\end{subtable}
\begin{subtable}{0.4\textwidth}
\caption{}
\label{table:2}
\centering
\scalebox{0.6} {
\begin{tabular}{@{}lrr@{}}
  \hline
r & $p_{r/n}^\text{Bon}$ & $p_{r/n}^{\text{new}}$ \\ 
  \hline
2 & 3.73E-04 & \textbf{4.49E-05} \\ 
  3 & 3.98E-04 & \textbf{4.66E-05} \\ 
  4 & 6.98E-04 & \textbf{7.50E-05} \\ 
  5 & 7.29E-04 & \textbf{1.18E-04} \\ 
  6 & 8.59E-04 & \textbf{1.31E-04} \\ 
  7 & 3.52E-03 & \textbf{1.39E-04} \\ 
  8 & 5.50E-03 & \textbf{4.23E-04} \\ 
  9 & 2.38E-02 & \textbf{1.90E-02} \\ 
  10 & 3.43E-02 & \textbf{2.66E-02} \\ 
  11 & 3.75E-02 & \textbf{2.81E-02} \\ 
  12 & 4.37E-02 & \textbf{4.63E-02} \\ 
  13 & \textbf{5.56E-02} & 6.45E-02 \\ 
  14 & 8.07E-02 & \textbf{6.45E-02} \\ 
  15 & 8.56E-02 & \textbf{7.36E-02} \\ 
  16 & 2.35E-01 & \textbf{7.36E-02} \\ 
  17 & 2.11E-01 & 2.11E-01 \\ 
  18 & 9.64E-01 & 9.64E-01 \\ 
   \hline
\end{tabular}
}
\end{subtable}
\begin{subtable}{\textwidth}
\caption{}
\label{table:3}
\centering
\scalebox{0.6} {
\begin{tabular}{@{}lrrrrrrrr@{}}
  \hline
 & \textbf{Age} & \textbf{Sex} & \textbf{Diabetes} & 
 \textbf{Stroke or TIA} & \textbf{Creatinine} & \textbf{$\text{CHADS}_2$} & 
 \textbf{VKA} & \textbf{TTR} \\ 
\hline
 $H_0^{1/k_i}$& 9.37E-06 & 1.74E-05 & 1.64E-05 & 1.47E-05 & 5.83E-06 & 5.29E-05 & 5.61E-06 & 1.98E-06 \\ 
 $H_0^{2/k_i}$& 9.26E-03 & 2.38E-03 & 3.81E-03 & 2.14E-02 & 3.68E-02 & 4.78E-02 & 1.61E-02 & 4.68E-03 \\ 
 $H_0^{3/k_i}$& -- & -- & -- & -- &  9.64E-01 & 1.05E-01 &-- & -- \\ 
   \hline
\end{tabular}
}
\end{subtable}
\caption{(a) The original data and individual p-values: the blue color highlights 
individual p-values that are not significant at level $\alpha = 0.05$; (b) 
the values of Bonferroni BHPC p-value and the new GBHPC p-values when $r$ changes 
in $H_{0}^{r/n}$; (c) the grouping factor level combined p-values: here 
$k_i = |I_i|$ and each p-value is the Fisher's BHPC p-value on $p_{I_i}$ testing 
for $H_{0}^{r_1/k_i}$ where $r_1 = 1, 2, 3$.} 
\label{table:123}
\end{table}

\section{Conclusion and future work}\label{sec:conclusion}
Partial conjunction hypotheses are natural hypotheses to test for 
measuring repeated effects across settings/studies. The null is rejected 
only when at least $r$ hypotheses are non-null. By testing PC hypotheses at 
different $r$ values, one can also construct a confidence interval of $r_0/n$, 
the true proportion of non-null hypotheses.

This paper characterizes the admissible p-values for a partial conjunction test
of independent hypotheses or hypotheses with positively dependent P-values, 
within the class of non-decreasing 
p-values. Any monotone admissible p-value for
$H_0^{r/n}$ is the maximum of the non-decreasing p-values for the global null in each combination 
of $n-r+1$ hypotheses, which we call GBHPC p-values.
We have shown that for sensitive GBHPC p-values, as long as each meta-analysis 
p-value of the $n-r+1$ hypotheses is admissible,
the combined p-value is monotone admissible. 
A consequence is that among
combined p-values that only depend on the order statistics of individual p-values, 
the original BHPC p-values 
are the only monotone admissible ones. 
We also have found inadmissibility of GBHPC p-values 
without the monotonicity constraint. However, 
the dominated tests only have a moderate power gain at low power regions 
in the alternative space. 
Since these counter-examples are not monotone, they are 
hard to be explained in practice thus 
not reasonable choices. 

In summary, we illustrated the properties of tests for a PC hypothesis and 
characterized a class of good tests called GBHPC p-values. 
Compared with its symmetric form, the BHPC p-values, GBHPC p-values 
have more flexibility to 
adapt to complicated problem structure, 
thus can have power gain at important regions in the 
alternative space, as we showed in our simulations and real data examples. 
The computational cost of non-symmetric GBHPC p-values can be of a concern, 
but there are special cases where GBHPC p-values are computable. 
One of the future directions is to expand the applications where computable 
GBHPC p-values are available.

One other direction is to understand properties of 
the confidence interval of $r$ constructed 
by GBHPC p-values. In \Cref{sec:real}, the result showed that though the newly proposed 
GBHPC p-values can be much smaller than Simes or Bonferroni BHPC p-values 
at many $r$ values, the confidence interval that the two methods constructed are 
still the same. 
Finally, there are variations of partial conjunctions that are useful in practice. 
For example, the count of replicability may vary for different hypotheses. Replication 
of effects from two distinct classes can be of more interest than replication in two 
similar classes. Another variation is to require that 
a null hypothesis is rejected only when there are at least $r$ non-nulls with 
the same sign of effect. Such hypotheses can have very complex alternative and null 
space, and it can be the future work to understand their properties.



\bibliographystyle{chicago}
\bibliography{ref}

\end{document}